%% file: Ishida_quasi.tex
\documentclass{amsart}
\usepackage{amsmath, amssymb, amscd, amsbsy, amsthm, righttag}
\usepackage{graphicx}
\newtheorem{theorem}{Theorem}[section]
\newtheorem{lemma}[theorem]{Lemma}
\newtheorem{prop}[theorem]{Proposition}
\newtheorem{cor}[theorem]{Corollary}
\theoremstyle{definition}

\theoremstyle{remark}
\newtheorem{remark}[theorem]{Remark}
\numberwithin{equation}{section}

\def\Z{{\mathbb Z}}

\def\R{{\mathbb R}}

\def\aread{{\rm Diff}_\Omega^\infty (D^2, \partial D^2)}
\def\areas{{\rm Diff}_\Omega^\infty (S^2)_0}
\def\sym{{\mathfrak S}}
\def\trans{{\mathcal T}}

\begin{document}
\title[Quasi-morphisms on the group of area-preserving diffeomorphisms]
{Quasi-morphisms on the group of area-preserving diffeomorphisms of the $2$-disk via braid groups}
\author{Tomohiko Ishida}
\address{Graduate School of Mathematical Sciences,
University of Tokyo, 3-8-1 Komaba,
Meguro-ku, Tokyo 153-8914, Japan.}
\email{ishidat@ms.u-tokyo.ac.jp}
\date{\today}
\begin{abstract}
Recently Gambaudo and Ghys proved that 
there exist infinitely many quasi-morphisms 
on the group $\aread$ of area-preserving diffeomorphisms of the $2$-disk $D^2$.
For the proof, they constructed a homomorphism 
from the space of quasi-morphisms on the braid group 
to the space of quasi-morphisms on $\aread$.
In this paper, we study this homomorphism and prove its injectivity.
\end{abstract}

\maketitle
\section{Introduction}
For a group $G$, a function $\phi\colon G\to \R$ is called a {\it quasi-morphism}
if the real valued function on $G\times G$ defined by
\[ (g, h)\mapsto \phi (gh)-\phi (g)-\phi (h) \]
is bounded.
The real number
\[ D(\phi ) =\sup _{g, h\in G} |\phi (gh)-\phi (g)-\phi (h)| \]
is called the {\it defect} of $\phi$.
We denote the $\R$-vector space of quasi-morphisms 
on the group $G$ by $\hat{Q} (G)$.
By definition, bounded functions on groups are quasi-morphisms.
Hence we denote the set of bounded functions on the group $G$ by $C_b^1(G; \R)$ 
and consider the quotient space $Q(G)=\hat{Q}(G)/C_b^1(G; \R)$.
A quasi-morphism $\phi\colon G\to\R$ is said to be {\it homogeneous} if the equation
\[ \phi (g^p)=p\ \phi (g) \]
holds for any $g\in G$ and $p\in\Z$.
For any quasi-morphism $\phi$, 
a homogeneous quasi-morphism $\tilde{\phi}$ is defined by setting
\[ \tilde{\phi}(g)=\lim _{p\to\infty}\frac{1}{p}\phi (g^p). \]
The limit always exists for each element $g$ of $G$.
The new function $\tilde{\phi}$ is in fact a quasi-morphism 
equal to the original quasi-morphism $\phi$ as an element of $Q(G)$.
Thus we can identify the vector space of homogeneous quasi-morphisms 
on the group $G$ with $Q(G)$.
Homogeneous quasi-morphisms are invariant under conjugations.
Therefore we are interested in $Q(G)$ rather than $\hat{Q}(G)$.

Let $\aread$ be the group of area-preserving $C^\infty$-diffeomorphisms 
of the $2$-disk $D^2$,
which are the identity on a neighborhood of the boundary.
On the vector space $Q(\aread)$, the following theorem is known.
\begin{theorem}
[Entov-Polterovich \cite{ep03}, Gambaudo-Ghys \cite{gg04}]\label{gg}
The vector space \\
$Q(\aread )$ is infinite dimensional.
\end{theorem}
To prove Theorem \ref{gg}, 
Entov and Polterovich explicitly constructed
uncountably many quasi-morphisms on $\aread$, which are linearly independent. 
After that Gambaudo and Ghys constructed 
countably many quasi-morphisms on $\aread$ by a different idea, 
which is to consider the suspension of area-preserving diffeomorphisms of the disk 
and average the value of the signature of the braids appearing in the suspension. 
By generalizing their strategy Brandenbursky \cite{brandenbursky11} defined the homomorphism
\[ \Gamma _n\colon Q(P_n(D^2))\to Q(\aread ), \]
which we review in Section \ref{proof}.
Here, $P_n(D^2)$ denotes the pure braid group on $n$-strands.

Let $B_n(D^2)$ be the braid group on $n$-strands.
The natural inclusion $i\colon P_n(D^2)\to B_n(D^2)$ induces the homomorphism
$Q(i)\colon Q(B_n(D^2))\to Q(P_n(D^2))$. 
In this paper, we study the homomorphism $\Gamma_n$ and prove the following theorem.
\begin{theorem}\label{inj}
The composition
\[ \Gamma _n\circ Q(i)\colon Q(B_n(D^2))\to Q(\aread ) \]
is injective.
\end{theorem}

\vskip 5pt
\noindent \textbf{Acknowledgments.}
This work is the main part of the author's doctoral thesis at University of Tokyo, 
under the supervision of Professor Takashi Tsuboi.
The author wishes to thank him for many helpful advices.
The author also thanks to Professors \'Etienne Ghys and Shigeyuki Morita 
for their warmly encouragement 
and to Professor Shigenori Matsumoto for valuable suggestions.
The author is very grateful to the referee 
for careful reading and pointing out errors in the manuscript.
The author is supported by JSPS Research Fellowships
for Young Scientists (23$\cdot$1352).

\section{Gambaudo and Ghys' construction and proof of the main theorem}\label{proof}

\begin{sloppypar}
In this section, we review Gambaudo and Ghys' construction \cite{gg04}
of quasi-morphisms on the group $\aread$ in a generalized form 
and prove Theorem \ref{inj}.
\end{sloppypar}

\begin{sloppypar}
Let $X_n(D^2)$ be the configuration space of ordered $n$-tuples in the $2$-disk $D^2$
and $x^0=(x _1^0, \dots , x_n^0)$ its base point.
For any $g\in\aread$ and for almost alls $x=(x _1, \dots , x_n)\in X_n(D^2)$, 
we define the pure braid $\gamma (g; x)$ as the following.
First we set the loop $l(g; x)\colon [0, 1]\to X_n(D^2)$ by
\[ l(g; x)(t)=\left\{ \begin{array}{ll}
\{ (1-3t)x_i^0+3tx_i\} &\displaystyle (0\leq t\leq\frac{1}{3}) \\[.6em]
\{ g_{3t-1}(x_i)\} &\displaystyle (\frac{1}{3}\leq t\leq\frac{2}{3})\\[.6em]
\{ (3-3t)g(x_i)+(3t-2)x_i^0\} &\displaystyle (\frac{2}{3}\leq t\leq 1) \\
\end{array}\right. ,\]
where $\{ g_t\}_{t\in[0, 1]}$ is a Hamiltonian isotopy such that $g_0$ is the identity and $g_1=g$.
We define the pure braid $\gamma (g; x)$ to be the braid represented by the loop $l(g; x)$.
For almost every $x$, the braid $\gamma (g; x)$ is well-defined. 
Furthermore, the braid $\gamma (g; x)$ is independent of the choice of the flow $\{ g_t\}$.
This is because of the fact the group $\aread$ is contractible, 
which is easily proved from the contractibility 
of the diffeomorphism group ${\rm Diff}^\infty (D^2, \partial D^2)$ of $D^2$ \cite{smale59}
and the homotopy equivalence between ${\rm Diff}^\infty (D^2, \partial D^2)$ and $\aread$ \cite{moser65}.
For a quasi-morphism $\phi$ on the pure braid group $P_n(D^2)$ on $n$-strands, 
we define the function $\hat{\Gamma}_n(\phi )\colon\aread\to\R$ by
\[ \hat{\Gamma}_n(\phi)(g)=\int _{x\in X_n(D^2)}\phi (\gamma (g; x))dx. \]
For any $\phi\in Q(P_n(D^2))$ and $g\in\aread$ the function $\phi (\gamma (g; \cdot ))$ is integrable 
and thus the map $\hat{\Gamma}_n\colon \hat{Q}(P_n(D^2))\to \hat{Q}(\aread)$ is well-defined \cite{brandenbursky11}.
The obtained function $\hat{\Gamma}_n(\phi )\colon\aread\to\R$ is also a quasi-morphism 
and the map $\hat{\Gamma}_n\colon \hat{Q}(P_n(D^2))\to \hat{Q}(\aread)$ 
is clearly $\R$-linear. 
Moreover, 
it is easily checked that any bounded function on $P_n(D^2)$ is mapped 
to a bounded function on $\aread$
and thus the homomorphism 
$\hat{\Gamma}\colon \hat{Q}(P_n(D^2)) \to \hat{Q}(\aread)$ 
induces the homomorphism $\Gamma _n \colon Q(P_n(D^2)) \to Q(\aread)$. 
\end{sloppypar}

\begin{remark}\label{homo}
\begin{sloppypar}
It is easy to see 
that the homomorphism $\Gamma _n \colon Q(P_n(D^2)) \to Q(\aread)$ 
maps the classical linking number homomorphism 
${\rm lk}_n\colon B_n(D^2)\to \R$ on the braid group 
to a homomorphism on $\aread $. 
In fact, 
the image of ${\rm lk}\colon B_n(D^2)\to \R$ 
by the homomorphism $\Gamma _n({\rm lk}_n)$
coincides with a constant multiple of the classical Calabi homomorphism on $\aread$ \cite{gg97}
and in this sense quasi-morphisms obtained in this way can be considered 
as generalizations of the Calabi homomorphism.
By an argument of Brandenbursky, 
which verify that the homomorphism 
$\Gamma\colon \hat{Q}(P_n)\to \hat{Q}(\aread )$ 
is well-defined, 
it is observed that quasi-morphisms obtained 
by the homomorphism $\hat{\Gamma}_n \colon \hat{Q}(P_n(D^2)) \to \hat{Q}(\aread)$
can be defined on the group 
of area-preserving $C^1$-diffeomorphisms of $D^2$,
as well as the Calabi homomorphism. 
\end{sloppypar}
\end{remark}

Now we are ready to prove Theorem \ref{inj}.

\begin{proof}[Proof of Theorem \ref{inj}]
Let us suppose that a homogeneous quasi-morphism $\phi\in\hat{Q}(B_n(D^2))$ is non-trivial.
Then there exists a braid $\beta\in B_n(D^2)$ such that $\phi (\beta )\neq 0$.
We may assume that $\beta$ is pure.
It is sufficient to prove 
that the homogeneous quasi-morphism $\hat{\Gamma}_n(\phi )\in\hat{Q}(\aread )$ is also non-trivial.
That is, there exists an area-preserving diffeomorphism $g\in\aread$ such that 
\[ \lim_{p\to\infty}\frac{1}{p}\Gamma _n(\phi)(g^p)\neq 0. \]

Let $A_{i, j}$ be the pure braid 
which twists only the $i$-th and the $j$-th strands for $1\leq i<j\leq n$ (see Figure \ref{pure}).
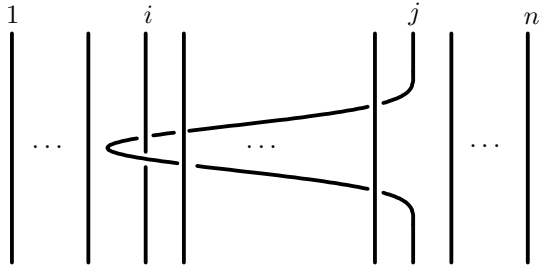
\begin{figure}[htbp]
\begin{center}
\input{pure}
\end{center}
\caption{pure braid $A_{i, j}$}
\label{pure}
\end{figure}
Since the braid $\beta$ is pure, 
it can be written as a composition of $A_{i, j}$'s and their inverses.
We take $n$ disjoint subsets $U_i$'s of $D^2$.
Furthermore, for a pair of $(i, j)$, 
we take subsets $V_{i, j}$ and $W_{i, j}$ of $D^2$ 
such that $U_i\cup U_j\subset W_{i, j}\subset V_{i, j}$, 
$U_k\cap V_{i, j}=\emptyset$ if $k\neq i, j$ 
and $V_{i, j}, W_{i, j}$ are diffeomorphic to $D^2$. 
Let $\{ h_t\}_{t\in[0, 1]}$ be a path in $\aread$ 
such that the support of $h_t$ is contained in the interior of $V_{i, j}$ 
and rotates $W_{i, j}$ once.  
Taking paths $\{ h_t\}$'s constructed above for the all $A_{i, j}$'s 
which present $\beta$ and composing them, 
we have a path $\{ g_t\}_{t\in[0, 1]}$ in $\aread$ with $g_0=id$ 
which twists $U_i$'s in the form of the pure braid $\beta$.
If we set $g=g_1$, then $g$ is the identity on $U_i$'s 
and $\gamma (g; (x_1, \dots , x_n))=\beta$ for $x_i\in U_i$.
Then by setting $U=U_1\cup\dots\cup U_n$, we have
\begin{align*}
&\quad \lim_{p\to\infty}\frac{1}{p}\hat{\Gamma} _n(\phi )(g^p) \\
&=\lim _{p\to\infty}\frac{1}{p}
\left(\int _{x\in X_n(U)}\phi (\gamma (g^p; x))dx
+\int _{x\in X_n(D^2)\setminus X_n(U)}\phi(\gamma (g^p ; x))dx\right) \\
&=\int _{x\in X_n(U)}\phi (\gamma (g; x))dx
+\lim _{p\to\infty}\frac{1}{p}\int _{x\in X_n(D^2)\setminus X_n(U)}\phi (\gamma (g^p ; x))dx .
\end{align*}
If we denote the first term of the equation by $Y$ 
and set $a_i={\rm area}(U_i)$ and $[n]=\{ 1, \dots , n\}$,
then $Y$ is written as
\[ \int _{x\in X_n(U)}\phi (\gamma (g; x))dx=\sum _{F\colon [n]\to [n]} \left( \prod _{i=1}^na_{F(i)}\right) x_F, \]
where $x_F=\phi (\gamma _F)$ and 
$\gamma _F=\gamma (g; x)$ for $x$ in the case when each $x_i$ is in $U_{F(i)}$.
The real numbers $x_F$'s have the following properties.
\begin{enumerate}
 \item[(i)] For two maps $F\text{ and }G\colon [n]\to [n]$, 
if $\# F^{-1}(i)=\# G^{-1}(i)$ for each $1\leq i\leq n$ then $x_F=x_G$.
 \item[(ii)]  If a map $F\colon [n]\to [n]$ is bijective, then $x_F$ is non-zero.
\end{enumerate}
The property (i) follows from the invariance of $\phi$ under conjugation
and the property (ii) follows because $\phi (\beta )$ is non-zero. 
Therefore, the coefficient of $a_1\dots a_n$ in $Y$ is non-zero.
Since the polynomial $Y$ is not identically $0$, 
we can choose $a_i$'s s that $Y$ is non-zero.

Note that if we replace $a_i$'s by bigger ones fixing the ratio of any two of them
the term $Y$ stays non-zero.
On the other hand, the values $\phi (\gamma (g ; x))$ is bounded 
because of the construction of $g$, and we thus have 
\[ \lim _{p\to\infty}\frac{1}{p}\int _{x\notin X_n(U)}\phi (\gamma (g^p ; x))dx \to 0
\quad\text{(as}\quad a_1+\dots +a_n\to{\rm area}(D^2)\text{)}. \]
This completes the proof.
\end{proof}

As we noted in Remark \ref{homo}, 
The homomorphism $\hat{\Gamma}_n$ maps any homomorphism on $P_n(D^2)$ 
to a homomorphism on $\aread$. 
Hence the homomorphism 
\[ Q(P_n(D^2))/H^1(P_n(D^2); \R)\to Q(\aread )/H^1(\aread ; \R) \]
is also induced.
By an argument similar to the proof of Theorem \ref{inj}, 
the following proposition holds.
\begin{prop}
The map
\[ Q(B_n(D^2))/H^1(B_n(D^2); \R)\to Q(\aread)/H^1(\aread ; \R) \]
induced by the composition $\Gamma _n\circ Q(i)\colon Q(B_n(D^2))\to Q(\aread )$ is injective.
\end{prop}

The homomorphism $\Gamma _n \colon Q(P_n(D^2))\to Q(\aread )$ 
can be defined also for the $2$-sphere $S^2$ instead of $D^2$
as Gambaudo and Ghys mentioned in their paper.
Let $\areas$ be the identity component of the group of area-preserving diffeomorphisms of $S^2$. 
Then we can choose a pure braid $\gamma (g; x)\in P_n(S^2)$
for any $g\in\areas$ and for almost every $x\in X_n(S^2)$ 
as in the case of the $2$-disk.
Since the group $\areas$ is homotopy equivalent to $SO(3)$ \cite{moser65}\cite{smale59}
and its fundamental group has order $2$, 
for any element $g$ of $\areas$
there exist two homotopy classes of paths 
connecting the identity and $g$ in $\areas$.
However, for any homogeneous quasi-morphism $\phi$ on $P_n(S^2)$, 
the value $\phi (\gamma (g; x))$ is independent of the choice of the path. 
In fact, from a path which represents the generator of $\pi _1(\areas )$ 
has order $2$ and is in the center of $P_n(S^2)$.
Hence the homomorphism $\Gamma _n\colon Q(P_n(S^2))\to Q(\areas )$ is defined.
Since the braid group $B_n(S^2)$ of the $2$-sphere on $n$-strands 
can be considered as a quotient group of the braid group $B_n(D^2)$, 
by an argument similar to the proof of Theorem \ref{inj}, 
we obtain the following theorem.

\begin{theorem}\label{sphere}
The composition
\[ \Gamma _n\circ Q(i)\colon Q(B_n(S^2))\to Q(\areas ) \]
is injective.
\end{theorem}

The homomorphism $Q(i)$ in the statement of Theorem \ref{sphere} 
is the one induced from the inclusion $i\colon P_n(S^2)\to B_n(S^2)$.

\section{Kernel of the homomorphism $\Gamma _n$}

The homomorphism $\Gamma _n\colon Q(P_n (D^2))\to Q(\aread )$ itself is not injective
although Theorem \ref{inj} holds.
In this section we study the kernel of the homomorphism $\Gamma _n$.

Let $G$ be a group and $H$ its finite index subgroup. 
We denote by $\overline{\beta}$ the image of an element $\beta\in G$ 
by the natural projection $G\to G/H$.
For each left coset $\sigma\in G/H$ of $G$ modulo $H$, 
we fix an element $\gamma _\sigma\in G$ such that $\overline{\gamma _\sigma}=\sigma$
and for any $\phi\in\hat{Q}(H)$ define the function $\hat{\trans}(\phi )\colon G\to\R$ by
\[ \hat{\trans}(\phi )(\beta )=
\frac{1}{(G:H)}\sum _{\sigma\in G/H}
\phi ({\gamma _{\overline{\beta\gamma _{\sigma}}}}^{-1} \beta\gamma _{\sigma}). \]
Since ${\gamma _{\overline{\beta\gamma _{\sigma}}}}^{-1} \beta\gamma _{\sigma}$ 
is in $H$, 
the function $\hat{\trans}(\phi )$ is well-defined on $G$.
\begin{lemma}\label{qhat}
For any quasi-morphism $\phi$ on $H$, 
the function $\hat{\trans}(\phi)\colon G\to\R$ is also a quasi-morphism.
\end{lemma}

\begin{proof}
Since the equality 
\[ {\gamma _{\overline{\beta _1\beta _2\gamma _\sigma}}}^{-1}
\beta _1\beta _2 \gamma _\sigma
=({\gamma _{\overline{\beta _1\beta _2\gamma _\sigma}}}^{-1}
\beta _1\gamma _{\overline{\beta _2\gamma _\sigma}})
({\gamma _{\overline{\beta _2\gamma _\sigma}}}^{-1}\beta _2\gamma _\sigma ) \]
holds, we have the inequality
\begin{align*}
&\quad |\hat{\trans}(\phi )(\beta _1\beta _2)-\hat{\trans}(\phi )(\beta _1)-\hat{\trans}(\phi )(\beta _2)| \\
&=\frac{1}{(G:H)}\Bigg|\sum _{\sigma\in G/H}\Big\{ 
\phi (({\gamma _{\overline{\beta _1\beta _2\gamma _\sigma}}}^{-1}
\beta _1\gamma _{\overline{\beta _2\gamma _\sigma}})
({\gamma _{\overline{\beta _2\gamma _\sigma}}}^{-1}\beta _2\gamma _\sigma )) \\
&\qquad\qquad\qquad\qquad\qquad 
-\phi ({\gamma _{\overline{\beta _1\gamma _{\sigma}}}}^{-1} \beta _1\gamma _{\sigma})
-\phi ({\gamma _{\overline{\beta _2\gamma _{\sigma}}}}^{-1} \beta _2\gamma _{\sigma}) \Big\} \Bigg| \\
&=\frac{1}{(G:H)}\Bigg|\sum _{\sigma\in G/H}\Big\{ 
\phi (({\gamma _{\overline{\beta _1\beta _2\gamma _\sigma}}}^{-1}
\beta _1\gamma _{\overline{\beta _2\gamma _\sigma}})
({\gamma _{\overline{\beta _2\gamma _\sigma}}}^{-1}\beta _2\gamma _\sigma )) \\
&\qquad\qquad\qquad\qquad\qquad 
-\phi ({\gamma _{\overline{\beta _1\beta _2\gamma _\sigma}}}^{-1}
\beta _1\gamma _{\overline{\beta _2\gamma _\sigma}})
-\phi ({\gamma _{\overline{\beta _2\gamma _{\sigma}}}}^{-1} \beta _2\gamma _{\sigma}) \Big\} \Bigg| \\
&\leq D(\phi ) . 
\end{align*}
Hence the function $\hat{\trans}(\phi )\colon G\to\R$ is also a quasi-morphism.
\end{proof}

The map $\hat{\trans}\colon \hat{Q}(H)\to\hat{Q}(G)$ is clearly $\R$-linear 
and induces a homomorphism $\trans \colon Q(P_n(D^2)) \to Q(B_n(D^2))$.
Furthermore, the following proposition holds.

\begin{prop}\label{well-defined}
The homomorphism $\trans \colon Q(H) \to Q(G)$ 
is independent of the choice of $\gamma _\sigma$'s.
\end{prop}

\begin{proof}
Suppose that $\phi$ is a homogeneous quasi-morphism on $H$.
If an element $\beta$ is in $H$, then $\overline{\gamma _\sigma \beta}=\sigma$
for each $\sigma\in G/H$. 
For any $\beta\in G$ there exists an integer $k$ 
such that $\beta ^k$ is in $H$ 
and we have
\begin{align}\label{indep}
\lim _{p\to\infty} \frac{1}{p}\hat{\trans}(\phi )(\beta ^p)
&=\lim _{p'\to\infty} \frac{1}{kp'}\hat{\trans}(\phi )(\beta ^{kp'}) \nonumber \\
&=\lim _{p'\to\infty} \frac{1}{(G:H)kp'}\sum _{\sigma\in G/H}
\phi (\gamma _\sigma ^{-1}\beta^k\gamma _\sigma )^{p'} \nonumber \\
&=\frac{1}{(G:H)k}\sum _{\sigma\in G/H}\phi (\gamma _\sigma ^{-1}\beta^k\gamma _\sigma ). 
\end{align}
Since $\phi$ is invariant under conjugations in $H$,
the value $\phi (\gamma _\sigma ^{-1}\beta^k\gamma _\sigma )$ depends only on $\sigma$.
\end{proof}

Let $Q(i)\colon Q(G)\to Q(H)$ be the homomorphism induced by 
the inclusion $i\colon H\to G$. 
As a corollary to Equality (\ref{indep}), we have the following.

\begin{cor}\label{id}
The composition $\trans\circ Q(i)\colon Q(G)\to Q(G)$
is the identity on $Q(G)$.
Furthermore, we have the decomposition
\[ Q(H)={\rm Ker}(\trans )\oplus{\rm Im}(Q(i)) \]
as vector spaces.
\end{cor}

\begin{remark}
Of course, the homomorphism $\hat{\trans}(\phi )\colon G\to\R$ 
can be defined using the right coset $H\backslash G$ instead of $G/H$ by
\[ \hat{\trans}(\phi )(\beta )=
\frac{1}{(G:H)}\sum _{\sigma\in G/H}
\phi (\gamma _\sigma\beta{\gamma _{\overline{\gamma _\sigma\beta}}}^{-1}). \]
By an argument similar to the proof 
of Lemma \ref{qhat} and Proposition \ref{well-defined}, 
it is verified that this alternative definition is also well-defined 
and induces the same homomorphism $\trans\colon Q(H) \to Q(G)$. 
\end{remark}

\begin{remark}
The homomorphism $\trans\colon Q(H)\to Q(G)$ 
is just a straightforward generalization of transfer map, 
and it is also introduced in \cite{malyutin09} and \cite{walker12}.
\end{remark}

Since the pure braid groups $P_n(D^2)$ and $P_n(S^2)$ 
are finite index subgroups of the braid groups $B_n(D^2)$ and $B_n(S^2)$, respectively, 
the homomorphisms 
\[ \trans \colon Q(P_n(D^2))\to Q(B_n(D^2))
\quad\text{ and }\quad
\trans \colon Q(P_n(S^2))\to Q(B_n(S^2)) \]
can be defined and Corollary \ref{id} is true 
for $G=B_n(D^2), H=P_n(D^2)$ and $G=B_n(S^2), H=P_n(S^2)$, respectively. 

The following proposition is the main result of this section. 

\begin{prop}\label{mean}
The composition 
\[ \Gamma _n\circ Q(i)\circ\trans\colon Q(P_n(D^2))\to Q(\aread ) \]
coincides with $\Gamma _n$.
In particular, ${\rm Ker}(\Gamma _n)={\rm Ker}(\trans )$ 
and ${\rm Im}(\Gamma _n)={\rm Im}(\Gamma _n\circ Q(i))$.
\end{prop}

\begin{proof}
Let $\sym _n$ be the symmetric group of $n$ symbols. 
By Equality (\ref{indep}), 
for any homogeneous quasi-morphism $\phi\in Q(P_n(D^2))$ 
and any area-preserving diffeomorphism $g\in\aread$, 
\begin{align}\label{composetrans}
\lim _{p\to\infty}\frac{1}{p}
\hat{\Gamma}_n\circ Q(i)\circ\hat{\trans} (\phi )(g^p)
=\frac{1}{n!}\sum _{\sigma\in\sym _n}
\lim _{p\to\infty}\frac{1}{p}
\int _{x\in X_n(D^2)}
\phi (\gamma _\sigma\gamma (g^p; x)\gamma ^{-1}_\sigma )dx. 
\end{align}

For any $\sigma\in\sym _n$ and almost all $x\in D^2$, 
we set the path $l\colon [0, 1]\to X_n(D^2)$ by
\[ l(t)=\left\{ \begin{array}{ll}
\{ (1-2t)x_i^0+2tx_i\} &\displaystyle (0\leq t\leq\frac{1}{2}) \\[.6em]
\{ (2-2t)x_i+(2t-1)x_{\sigma (i)}^0\} &\displaystyle (\frac{1}{2}\leq t\leq 1) \\
\end{array}\right. .\]
Considering the path $l$ as a loop in the quotient space $X_n(D^2)/\sym _n$, 
we define the braid $\beta (\sigma ; x)$ to be the braid represented by the loop $l$. 
Then by definition, 
\[ \beta (\sigma ; x)\gamma (g; \sigma ^{-1}(x))
\beta (\sigma ; g_\ast x)^{-1}=\gamma (g; x), \]
where the symmetric group $\sym _n$ acts on $X_n(D^2)$ by the permutation 
\[ \sigma (x_1, \dots , x_n)=(x_{\sigma (1)}, \dots , x_{\sigma (n)}). \]
Since the homomorphism $\trans\colon Q(P_n(D^2))\to Q(B_n(D^2))$ is defined 
independently to the choice of braids $\gamma _\sigma$'s, 
we may choose $\gamma _\sigma$ to be $\beta (\sigma ; x)$. 
Hence we have 
\begin{align*}
\gamma _\sigma\gamma (g; \sigma ^{-1}(x))\gamma _\sigma ^{-1}
&=\beta (\sigma ;x)\gamma (g; \sigma ^{-1}(x))\beta (\sigma ;x)^{-1} \nonumber \\
&=\gamma (g; x)\beta (\sigma ; g_\ast (x))\beta (\sigma ; x)^{-1}.
\end{align*}
Since the function $\phi(\beta (\sigma ; \cdot ))\colon D^2\to\R$ is bounded on $D^2$, 
we have 
\begin{align*}
&\quad \lim _{p\to\infty}\frac{1}{p}
\int _{x\in X_n(D^2)}
\phi (\gamma _\sigma\gamma (g^p; x)\gamma ^{-1}_\sigma )dx \\
&=\lim _{p\to\infty}\frac{1}{p}
\int _{x\in X_n(D^2)}
\phi (\gamma _\sigma\gamma (g^p; \sigma ^{-1}(x))\gamma ^{-1}_\sigma )dx \\
&=\lim _{p\to\infty}\frac{1}{p}
\int _{x\in X_n(D^2)}
\phi (\gamma (g^p; x))dx. 
\end{align*}
Therefore, by Equality (\ref{composetrans}), 
\begin{align*}
\lim _{p\to\infty}\frac{1}{p}
\hat{\Gamma}_n\circ Q(i)\circ\hat{\trans} (\phi )(g^p)
&=\frac{1}{n!}\sum _{\sigma\in\sym _n}
\lim _{p\to\infty}\frac{1}{p}
\int _{x\in X_n(D^2)}
\phi (\gamma (g^p; x))dx \\
&=\lim _{p\to\infty}\frac{1}{p}\hat{\Gamma}_n(\phi )(g^p)
\end{align*}
and thus we have $\Gamma _n\circ Q(i)\circ\trans=\Gamma _n$.

Then obviously ${\rm Ker}(\trans )\subseteq{\rm Ker}(\Gamma _n)$
and ${\rm Im}(\Gamma _n)={\rm Im}(\Gamma _n\circ Q(i))$ hold.
If $\phi\in{\rm Ker}(\Gamma _n)$ then 
\[ \Gamma _n\circ Q(i)\circ\trans (\phi)=\Gamma _n(\phi )=0 \]
and hence $\trans (\phi)=0$ by Theorem \ref{inj}.
Thus we have ${\rm Ker}(\Gamma _n)\subseteq{\rm Ker}(\trans )$.
\end{proof}

\begin{remark}
Proposition \ref{mean} also holds for 
$P_n(S^2)$ and $\areas$ instead of $P_n(D^2)$ and $\aread$, respectively.
\end{remark}

\bibliographystyle{amsplain}
\providecommand{\bysame}{\leavevmode\hbox to3em{\hrulefill}\thinspace}
\providecommand{\MR}{\relax\ifhmode\unskip\space\fi MR }
\providecommand{\MRhref}[2]{%
  \href{http://www.ams.org/mathscinet-getitem?mr=#1}{#2}
}
\providecommand{\href}[2]{#2}

\end{document}

%% file: pure.tex
\unitlength 0.1in
\begin{picture}( 27.3000, 14.3500)(  8.7000,-34.0000)
%
\special{pn 20}%
\special{pa 3000 3200}%
\special{pa 3000 3400}%
\special{fp}%
\special{pa 3000 2400}%
\special{pa 3000 2200}%
\special{fp}%
%
\special{pn 20}%
\special{pa 1600 3400}%
\special{pa 1600 2900}%
\special{fp}%
\special{pa 1600 2820}%
\special{pa 1600 2200}%
\special{fp}%
%
\special{pn 20}%
\special{pa 1800 3400}%
\special{pa 1800 2200}%
\special{fp}%
\special{pa 2800 3400}%
\special{pa 2800 2200}%
\special{fp}%
%
\special{pn 20}%
\special{pa 1828 2710}%
\special{pa 1912 2700}%
\special{pa 2042 2684}%
\special{pa 2086 2680}%
\special{pa 2214 2664}%
\special{pa 2298 2654}%
\special{pa 2340 2650}%
\special{pa 2380 2644}%
\special{pa 2418 2640}%
\special{pa 2456 2634}%
\special{pa 2528 2624}%
\special{pa 2562 2620}%
\special{pa 2596 2614}%
\special{pa 2626 2610}%
\special{pa 2656 2604}%
\special{pa 2686 2600}%
\special{pa 2712 2594}%
\special{pa 2738 2590}%
\special{pa 2762 2584}%
\special{fp}%
%
\special{pn 20}%
\special{pa 3000 3200}%
\special{pa 3000 3150}%
\special{pa 2996 3134}%
\special{pa 2994 3130}%
\special{pa 2994 3124}%
\special{pa 2992 3120}%
\special{pa 2986 3110}%
\special{pa 2978 3100}%
\special{pa 2962 3084}%
\special{pa 2954 3080}%
\special{pa 2946 3074}%
\special{pa 2938 3070}%
\special{pa 2928 3064}%
\special{pa 2916 3060}%
\special{pa 2904 3054}%
\special{pa 2876 3044}%
\special{pa 2860 3040}%
\special{pa 2844 3034}%
\special{fp}%
%
\special{pn 20}%
\special{pa 2762 3014}%
\special{pa 2738 3010}%
\special{pa 2712 3004}%
\special{pa 2686 3000}%
\special{pa 2656 2994}%
\special{pa 2626 2990}%
\special{pa 2596 2984}%
\special{pa 2562 2980}%
\special{pa 2528 2974}%
\special{pa 2456 2964}%
\special{pa 2418 2960}%
\special{pa 2380 2954}%
\special{pa 2340 2950}%
\special{pa 2298 2944}%
\special{pa 2214 2934}%
\special{pa 2086 2920}%
\special{pa 2042 2914}%
\special{pa 1912 2900}%
\special{pa 1870 2894}%
\special{fp}%
%
\special{pn 20}%
\special{pa 1766 2880}%
\special{pa 1726 2874}%
\special{pa 1688 2870}%
\special{pa 1652 2864}%
\special{pa 1616 2860}%
\special{pa 1584 2854}%
\special{pa 1554 2850}%
\special{pa 1526 2844}%
\special{pa 1500 2840}%
\special{pa 1478 2834}%
\special{pa 1458 2830}%
\special{pa 1440 2824}%
\special{pa 1426 2820}%
\special{pa 1406 2810}%
\special{pa 1402 2804}%
\special{pa 1400 2800}%
\special{pa 1402 2794}%
\special{pa 1406 2790}%
\special{pa 1414 2784}%
\special{pa 1426 2780}%
\special{pa 1440 2774}%
\special{pa 1458 2770}%
\special{pa 1478 2764}%
\special{pa 1500 2760}%
\special{pa 1526 2754}%
\special{pa 1554 2750}%
\special{fp}%
%
\special{pn 20}%
\special{pa 1640 2734}%
\special{pa 1674 2730}%
\special{pa 1710 2724}%
\special{pa 1748 2720}%
\special{fp}%
%
\special{pn 20}%
\special{pa 2844 2564}%
\special{pa 2860 2560}%
\special{pa 2876 2554}%
\special{pa 2904 2544}%
\special{pa 2916 2540}%
\special{pa 2928 2534}%
\special{pa 2938 2530}%
\special{pa 2946 2524}%
\special{pa 2954 2520}%
\special{pa 2962 2514}%
\special{pa 2968 2510}%
\special{pa 2978 2500}%
\special{pa 2986 2490}%
\special{pa 2992 2480}%
\special{pa 2994 2474}%
\special{pa 2994 2470}%
\special{pa 2996 2464}%
\special{pa 3000 2450}%
\special{pa 3000 2400}%
\special{fp}%
\put(21.2500,-28.0000){\makebox(0,0)[lb]{$\dots$}}%
\put(10.0500,-28.0000){\makebox(0,0)[lb]{$\dots$}}%
%
\special{pn 20}%
\special{pa 900 2200}%
\special{pa 900 3400}%
\special{fp}%
%
\special{pn 20}%
\special{pa 1300 2200}%
\special{pa 1300 3400}%
\special{fp}%
%
\special{pn 20}%
\special{pa 3200 2200}%
\special{pa 3200 3400}%
\special{fp}%
\put(32.9500,-27.9500){\makebox(0,0)[lb]{$\dots$}}%
%
\special{pn 20}%
\special{pa 3600 2200}%
\special{pa 3600 3400}%
\special{fp}%
\put(15.9000,-21.4500){\makebox(0,0)[lb]{$i$}}%
\put(29.7500,-21.5500){\makebox(0,0)[lb]{$j$}}%
\put(8.7000,-21.4500){\makebox(0,0)[lb]{$1$}}%
\put(35.8000,-21.5000){\makebox(0,0)[lb]{$n$}}%
\end{picture}%